\documentclass[review]{elsarticle}

\usepackage{amssymb, amsmath, amsthm, bm, 
	makecell, booktabs, diagbox, 
	algorithm, algpseudocode, 			  
	tikz, pgfplots, subfigure, adjustbox, 
	natbib}                                   

\usetikzlibrary{calc, matrix, shapes, arrows, positioning, fit, backgrounds}

\newtheorem{thm}{Theorem}
\newtheorem{cor}{Corollary}
\newtheorem{lem}{Lemma}

\numberwithin{equation}{section}

\usepackage{lineno, hyperref}
\usepackage{appendix}
\usepackage{cases}
\modulolinenumbers[5]

\journal{Advances in Applied Mathematics}

\begin{document}
\renewcommand{\qedsymbol}{}
\allowdisplaybreaks
\begin{frontmatter}

\title{Bijective enumeration of general stacks}

\author[mymainaddress]{Qianghui Guo\corref{mycorrespondingauthor}}
\cortext[mycorrespondingauthor]{Corresponding author}
\ead{guo@nankai.edu.cn}

\author[mymainaddress]{Yinglie Jin\corref{mycorrespondingauthor}}
\ead{yljin@nankai.edu.cn}

\author[mysecondaryaddress]{Lisa H. Sun\corref{mycorrespondingauthor}}
\ead{sunhui@nankai.edu.cn}

\author[mymainaddress]{Shina Xu}

\address[mymainaddress]{School of Mathematical Sciences, LPMC, Nankai University, Tianjin 300071, PR China.}
\address[mysecondaryaddress]{Center for Combinatorics, LPMC, Nankai University, Tianjin 300071, PR China.}

\begin{abstract}

Combinatorial enumeration of various RNA secondary structures and protein contact maps, is of great interest for both combinatorists and computational biologists. 
Enumeration of protein contact maps has considerable difficulties due to the significant higher vertex degree than that of RNA secondary structures. 
The state of art maximum vertex degree in previous works is two. 
This paper proposes a solution for counting stacks in protein contact maps with arbitrary vertex degree upper bound. 
By establishing bijection between such general stacks and $m$-regular $\Lambda$-avoiding $DLU$ paths, and counting the paths using theories of pattern avoiding lattice paths, we obtain a unified system of equations for generating functions of general stacks. 
We also show that previous enumeration results for RNA secondary structures and protein contact maps can be derived from the unified equation system as special cases.

 
\end{abstract}


\end{frontmatter}


\section{Introduction}
The diagram  $G([n], E)$ refering to a graph $G([n], E)$ represented by drawing $n$ vertices in a horizontal line and arcs $(i, j) \in E$ in the upper-half plane, is a classical combinatorial structure closely related to set partitions and lattice paths \citet{Chen2005, chen-linked-2008, stein-class-1978I}. It attracts extensive studies by various motivations, one of which is from computational molecular biology, where the diagram is often used to model biopolymer structures like RNA secondary structures and protein contact maps. 

Recall that RNA secondary structure is a sequence of nucleotides or bases which can form base pairs under the Watson-Crick pairing rule, an RNA secondary structure 
can be represented by a diagram, where any arc is a base pair. Since Waterman \cite{waterman1978secondary, WS78} set up a combinatorial framework for the study of RNA secondary structures in the 1970s, which has attracted significant interest from both combinatorialists and theoretical biologists \cite{howell1980computation, schmitt1994linear, stein1979rna, waterman-combinatorics-1979}.

Inspired by the enumeration of RNA secondary structure, the research of protein got more attention. Protein can be seen as a sequence of amino acid residues of twenty types. The function of a protein is directly determined by its three-dimensional structure, which researched by lattice model in general. In such model, protein fold are represented by self-avoiding walks on the particular lattice, see Istrail and Lam \cite{istrail-combinatorial-2009}.  Two residues are in a contact if they reside on two adjacent points in the lattice, but not consecutive in the sequence. All contacts of a protein fold form the contact map. The contact map for protein folding usually represented as the diagram.

Goldman et al. \cite{GIP99} showed that any protein contact map in the 2D square lattice can be decomposed into (at most) two stacks and one queue. Agarwal  et al. \cite{Aga} found that any protein contact map in the 3D cube lattice can be decomposed into (at most) $O(\sqrt{n})$ stacks or queues.
When folding a protein in the lattice model, different lattice models will introduce different degree and arc length constraints to the corresponding contact map. For instance, on the 2D square lattice, each internal vertex in the contact map has maximum degree 2, while the two terminal vertices can have maximum degree 3, and the arc length is at least 3. On the 3D cube lattice, the degree of each internal vertex and terminal vertex is at most 4 and 5, respectively, while the length of each arc is at least 3, see Figure \ref{fig:map}.

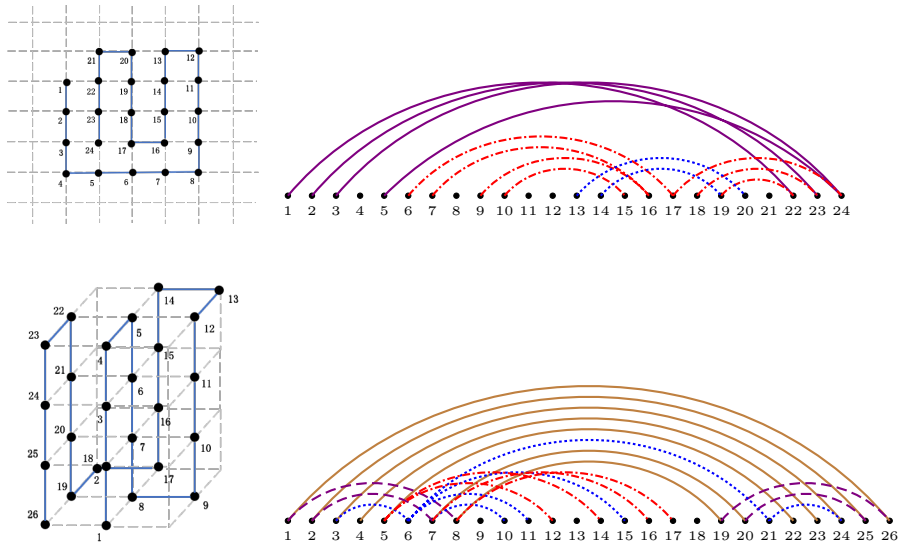
\begin{figure}[H]
\begin{minipage}[t]{0.3\textwidth}
\centering
 \includegraphics[scale=0.4]{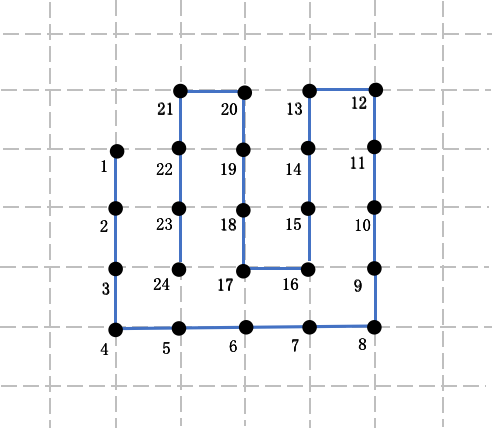}
\label{fig:side:a}
\end{minipage}%
\begin{minipage}[t]{0.4\textwidth}
\centering
\begin{tikzpicture}
			\foreach \s in {0.32} 
			{\foreach \x in {1, ..., 24}
				\filldraw (\x *\s, 0) circle[radius=1pt];
				\foreach \x in {1, ..., 24}
				\draw (\x*\s, -0.2) node{\tiny \x};
								
				\draw[thick, violet]
				(1*\s, 0) to [bend left=50] (22*\s, 0)
				(2*\s, 0) to [bend left=50] (23*\s, 0)
				(3*\s, 0) to [bend left=50] (24*\s, 0)
				(5*\s, 0) to [bend left=45] (24*\s, 0);
				\draw[thick, red, densely dash dot] 
				(6*\s, 0) to [bend left=50] (17*\s, 0)
				(7*\s, 0) to [bend left=50] (16*\s, 0)
				(9*\s, 0) to [bend left=50] (16*\s, 0)
				(10*\s, 0) to [bend left=50] (15*\s, 0)
				(14*\s, 0) to [bend left=50] (14*\s, 0)
				(17*\s, 0) to [bend left=50] (24*\s, 0)
				(18*\s, 0) to [bend left=50] (23*\s, 0)
				(19*\s, 0) to [bend left=50] (22*\s, 0);
				
				\draw[thick, blue, dash pattern=on 1pt off 1pt]
				(13*\s, 0) to [bend left=50] (20*\s, 0)
				(14*\s, 0) to [bend left=50] (19*\s, 0)
				(18*\s, 0) to [bend left=50] (18*\s, 0);
				}
	\end{tikzpicture}
\label{fig:side:b}
\end{minipage}
\begin{minipage}[t]{0.3\textwidth}
\centering
 \includegraphics[scale=0.5]{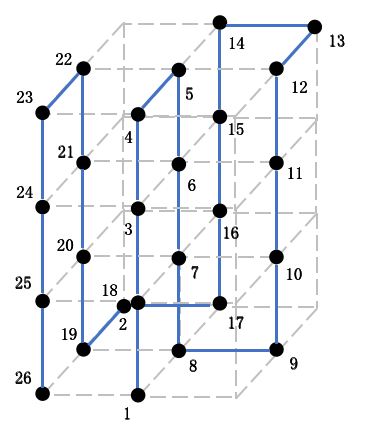}
\end{minipage}%
\begin{minipage}[t]{0.4\textwidth}
\centering
\begin{tikzpicture}
			\foreach \s in {0.32} 
			{\foreach \x in {1, ..., 26}
				\filldraw (\x *\s, 0) circle[radius=1pt];
				\foreach \x in {1, ..., 26}
				\draw (\x*\s, -0.2) node{\tiny \x};
								
				\draw[thick,brown]
				(1*\s, 0) to [bend left=50] (26*\s, 0)
				(2*\s, 0) to [bend left=50] (25*\s, 0)
				(3*\s, 0) to [bend left=50] (24*\s, 0)
				(4*\s, 0) to [bend left=50] (23*\s, 0)
				(5*\s, 0) to [bend left=50] (22*\s, 0)
				
				(7*\s, 0) to [bend left=50] (20*\s, 0)
				(8*\s, 0) to [bend left=50] (19*\s, 0);
				\draw[thick, violet, dash pattern=on 4pt off 2pt]
				(1*\s, 0) to [bend left=50] (8*\s, 0)
				(2*\s, 0) to [bend left=50] (7*\s, 0)
				(19*\s, 0) to [bend left=50] (26*\s, 0)
				(20*\s, 0) to [bend left=50] (25*\s, 0)
				;
				\draw[thick, blue,  dash pattern=on 1pt off 1pt] 
				(6*\s, 0) to [bend left=50] (21*\s, 0)
				(6*\s, 0) to [bend left=50] (11*\s, 0)
				(6*\s, 0) to [bend left=50] (15*\s, 0)
				(3*\s, 0) to [bend left=50] (6*\s, 0)
				(7*\s, 0) to [bend left=50] (10*\s, 0)
				(21*\s, 0) to [bend left=50] (24*\s, 0)
				;
				\draw[thick, red,densely dash dot]
				(5*\s, 0) to [bend left=50] (14*\s, 0)
				(5*\s, 0) to [bend left=50] (12*\s, 0)
				(7*\s, 0) to [bend left=50] (16*\s, 0)
				(8*\s, 0) to [bend left=50] (17*\s, 0);
				\draw[thick, yellow, dash pattern=on 6pt off 2pt]

				;

				}
	\end{tikzpicture}
\end{minipage}
\caption{Protein folds on 2D square and 3D cube lattices and their contact maps}
\label{fig:map}
\end{figure}

Istrail and Lam \cite{istrail-combinatorial-2009} proposed the question concerning generalizations of the Schmitt–Waterman counting formulas for RNA secondary structures \cite{schmitt1994linear} to enumerating protein stacks and queues.
In the context of combinatorics, a stack is a noncrossing diagram, and a queue is a nonnesting diagram\cite{Chen2014}. Moreover, a stack(queue) known as simple if the degree of each vertex bounded by one, and linear if the degree of each vertex bounded by two. A stack is said to be $m$-regular if the length of each arc is at least $m$. Clearly, RNA secondary structures can be viewed as $2$-regular simple stacks.

To explore the question, Waterman \cite{waterman1978secondary}, Howell et al. \cite{howell1980computation} and  Hofacker et al.\cite{hofacker1998combinatorics}, derived the recurrence relation of $m$-regular simple stack. Afterwards, 
Stein et al.\cite{stein1979rna} got the generating function of $m$-regular simple stack by its recurrence relation.Müller and Nebel \cite{Ne2015} presented the definition of extended RNA secondary structure ( $2$-regular linear stack), and obtained relevant quantitative analysis result. Chen et al. \cite{Chen2014} study the combinatorial enumeration of $m$-regular linear stack via the medium of zigzag stack and primary decomposition method, and got enumeration results in the form of generating function equation, recurrence relation and asymptotic formula. 
Furthermore, enumeration results, mainly in the form of generating functions, were obtained for extended $m$-regular linear stacks \cite{GS16}(the degree of terminal vertices can reach $3$); $m$-regular linear stacks with $n$ vertices and $k$ arcs \cite{guo2018combinatorics}; $2$-regular and $3$-regular simple queues \cite{GSW17}. 
 Guo et al. \cite{guo2023bijective} obtained explicit enumeration formulas for the saturated extended $2$-regular simple stacks (the degree of terminal vertices can reach $2$) by constructing a semi-bijective algorithm between such structures and the forests of small trees.

For breaking the restriction of degree to research the enumeration of stacks with larger degree of each vertex, we concerned with the stacks with the degree of each vertex bounded by $d$, defined as 
$d$-contact stacks. A stack is general if it is a $m$-regular $d$-contact stack with arbitrary $m,d\in\mathbb{Z}^+$. Denoting $S_{m,d}(x)$ as the generating function of $m$-regular $d$-contact stack. As illustrated in Figure \ref{fig:generalstack} is a $3$-regular $4$-contact stack. 
 \begin{figure}[H]
\begin{center}
\setlength{\unitlength}{0.4mm}
\begin{tikzpicture}
					\foreach \s in {1}{
						{\foreach \x in {1,...,8}
							\filldraw (\x *\s,0) circle[radius=.8pt];
							\foreach \x in {1,...,8}
							\draw (\x*\s,-0.2) node{\tiny \x};
							\foreach \start/\end in {1/3,1/8,3/5,3/8,5/8,6/8}
							\draw (\start*\s,0) to [bend left=60] (\end*\s,0);}}
				\end{tikzpicture}
  \caption[discard]{A $3$-regular $4$-contact stack} 
  \label{fig:generalstack}
\end{center}
\end{figure}
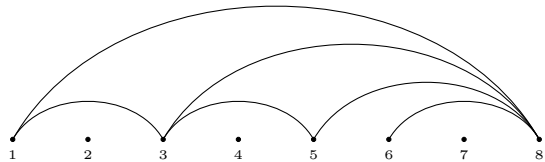

In order to enumerate the general stack, we established a bijection between $m$-regular $d$-contact stack and a specific Motzkin path with certain patterns avoiding, said as $m$-regular $\Lambda$-avoiding $DLU$ paths of length $dn$ from $(0, 0)$ to $(dn, 0)$. By decomposing the $m$-regular $\Lambda$-avoiding $DLU$ paths, we derive the following theorem to obtain $S_{m, d}(x)$ through $G^{\langle 0, 0\rangle }_{m, d}(x)$, where $G^{\langle s, t\rangle }_{m, d}(x) $ represented the generating function of  $m$-regular $\Lambda$-avoiding $DLU$ paths of length $dn$ from $(0, s)$ to $(dn, t)$.

\begin{thm}
\label{thm:syseq}
The generating function of $m$-regular $d$-contact stacks satisfying following eqation, 
 \begin{equation}
\label{equ:samenum}
S_{m, d}(x^d)=G^{\langle 0, 0\rangle }_{m, d}(x), 
\end{equation}
while $G^{\langle 0, 0\rangle }_{m, d}(x)$ can be determined by the following system of equations, 
\label{thm:equations}
\begin{numcases}{}
  G^{\langle 0, 0\rangle }_{m, d}=1+x^d\sum\limits_{i=0}^{d}G^{\langle i, 0\rangle }_{m, d}, \nonumber\\
  G^{\langle s, t\rangle }_{m, d}=(-1)^t\delta_{s, t}C_{m, d}+\sum\limits_{i=0}^{t-1}(-1)^{t-i+1}G^{\langle s-t+i, i\rangle }_{m, d}\label{equ:syseq} \\
  ~~~~~~~~+x^d\sum\limits_{a=1}^{d}\bigg((-1)^a\delta_{s, a}C_{m, d}+\sum\limits_{j=0}^{v-1}(-1)^{v-j+1}G^{\langle u-v+j, j\rangle }_{m, d}\bigg)\sum\limits_{c=0}^{d-a}G^{\langle c, t\rangle }_{m, d}, \nonumber\\
   ~~~~~~~~~~~~~~~~~~~~~~~~~~~~~~~~~~~~~~~~~~~~~~~~0\leq t \leq d-1, ~0<s\leq d, ~t<s, \nonumber
 \end{numcases}
where $G, C$ are short for $G(x), C(x)$, $u=\max(s, a), ~v=\min(s, a)$, $\delta_{i, j}=1~(i=j)$ or $0~(i\neq j)$, and $C_{m, d}(x)=\frac{x^{d(m-1)}-1}{x^d-1}$.
\end{thm}

This paper is organized as follows. In Section 2, we describe the bijection between general stacks and $DLU$ paths. In Section 3, we decompose the $DLU$ paths to get the system of equations \eqref{equ:syseq}.  In Section 4, we apply  Theorem \ref{thm:syseq} to verify the generating function of $m$-regular $1$-contact (simple) stack and $2$-contact (linear) stack. In addition, we derive the equation that the generating function of $m$-regular $3$-contact stack satisfying.

\section{The Bijection between general stacks and $DLU$ paths}

Recall that a Motzkin path of length $n$ is a lattice path in $\mathbb{Z}^{2}$ from $(0, 0)$ to $(n, 0)$ consisting of up-steps $(1, 1)$, down-steps $(1, -1)$ and horizontal steps $(1, 0)$ which never passes below $x$-axis. For convenience of description, we denote up-steps as $U$, down-steps as $D$, horizontal steps  as $L$. 

At the moment, we present $DLU$ path, which is the specific Motzkin path concerned in this paper. A $DLU$ path $M=P_1P_2\ldots P_n$ of length $dn$ means that for any $1\leq i\leq n$, the length of subpath $P_i$ is $d$, and $$P_i \in\{D^aL^bU^c| a+b+c=d, ~a, b, c\in\mathbb{N}\}, $$ and $M$ never passes below $x$-axis. Call $P_i$ a piece of path $M$. The height of a vertex $(x, y)$ is defined to be $y$-coordinate, and the height of a path is the minimal height of its vertices. A $DLU$ path $M$ of length $dn$ starting and ending at $x$-axis  is $m$-regular if for any subpath $UPD$ of $M$ such that
$$ P \notin \{\emptyset, L^d, L^{2d}, \ldots, L^{(m-2)d}\}.$$ 
A $DLU$ path $M=P_1P_2 P_3\ldots $  starting and ending at $x$-axis has pattern $\Lambda$ if there exists a subpath $UUPDD$ of $M$ with $P$ starting and ending at the same height $h$ and never passing below the height $h$, where $UU\subset P_i, DD\subset P_j$. 
Denoting the set of $m$-regular $\Lambda$-avoiding $DLU$ paths of length $dn$ from $(0, 0)$ to $(dn, 0)$ as $\mathcal{G}^{\langle 0, 0\rangle }_{m, d}(dn)$.
\begin{figure}[H]
\begin{center}
\setlength{\unitlength}{0.36mm}
\begin{picture}(305, 62)

\put(0, 20){\circle*{1.5}}
\put(10, 30){\circle*{1.5}}
\put(20, 40){\circle*{1.5}}
\put(30, 50){\circle*{1.5}}
\put(40, 40){\circle*{1.5}}
\put(50, 50){\circle*{1.5}}
\put(60, 60){\circle*{1.5}}
\put(70, 50){\circle*{1.5}}
\put(80, 40){\circle*{1.5}}
\put(90, 40){\circle*{1.5}}
\put(100, 30){\circle*{1.5}}
\put(110, 20){\circle*{1.5}}
\put(120, 20){\circle*{1.5}}
\put(0, 20){\line(1, 1){30}}
\put(30, 50){\line(1, -1){10}}
\put(40, 40){\line(1, 1){20}}
\put(60, 60){\line(1, -1){20}}
\put(80, 40){\line(1, 0){10}}
\put(90, 40){\line(1, -1){20}}
\put(110, 20){\line(1, 0){10}}
\multiput(15, 40)(1, 0){80}{\line(0, 1){0.5}}
\multiput(55, 60)(1, 0){10}{\line(0, 1){0.5}}
\put(49, 0){(a)}

\put(150, 20){\circle*{1.5}}
\put(160, 20){\circle*{1.5}}
\put(170, 30){\circle*{1.5}}
\put(180, 40){\circle*{1.5}}
\put(190, 40){\circle*{1.5}}
\put(200, 40){\circle*{1.5}}
\put(210, 50){\circle*{1.5}}
\put(220, 40){\circle*{1.5}}
\put(230, 40){\circle*{1.5}}
\put(240, 40){\circle*{1.5}}
\put(250, 30){\circle*{1.5}}
\put(260, 30){\circle*{1.5}}
\put(270, 40){\circle*{1.5}}
\put(280, 30){\circle*{1.5}}
\put(290, 20){\circle*{1.5}}
\put(300, 20){\circle*{1.5}}
\put(200, 40){\line(1, 1){10}}
\put(210, 50){\line(1, -1){10}}
\put(220, 40){\line(1, 0){20}}
\put(240, 40){\line(1, -1){10}}
\put(250, 30){\line(1, 0){10}}
\put(260, 30){\line(1, 1){10}}
\put(270, 40){\line(1, -1){20}}
\put(290, 20){\line(1, 0){10}}
\put(150, 20){\line(1, 0){10}}
\put(160, 20){\line(1, 1){20}}
\put(180, 40){\line(1, 0){20}}
\multiput(175, 40)(1, 0){100}{\line(0, 1){0.5}}

\put(216, 0){(b)}
\end{picture}
 \caption[discard]{(a) A path has pattern $\Lambda$, (b) a path avoids pattern $\Lambda$.} 
\end{center}
\end{figure}
 
Chen \cite{Chen2005} showed a classical bijection between noncrossing partial matching and Motzkin paths. To describe this bijection, for a vertex $v$ in a diagram, we called the left degree of $v$, denoted by $\mathrm{ldeg}(v)$, is the number of vertices $u$ such that $u<v$ and $(u, v)$ is an arc of the diagram. The right degree $\mathrm{rdeg}(v)$ is defined similarity. Clearly, for any vertex in the noncrossing partial matching, both left degree and right degree bounded by $1$. For a given noncrossing partial matching $M$, its corresponding Motzkin path $\theta(M)$ is defined by converting each vertex with right degree one into a step $U$, each vertex with left degree one into a step $D$, and each isolate vertex into a step $L$. Conversely, from a Motzkin path $P=p_1p_2\ldots p_n$ of length $n$, we can recover a noncrossing partial matching by the inverse of $\theta$. First, draw $n$ vertices $1, 2, \ldots, n$ from left to right on a horizontal line. Scanning up-step from right to left, and attaching each element $i$ to the leftmost available element of down-step to its right.

Denote the set of $m$-regular $d$-contact stacks of $[n]$ as $\mathcal{S}_{m, d}(n)$, its cardinality is denoted by $s_{m, d}(n)$, and the generating function is $$S_{m, d}(x) =\sum_{n\geq 0}s_{m, d}(n)x^{n}.$$ And then, we promoted the bijection between  noncrossing partial matching and Motzkin path, to build a bijection $\eta$ between $m$-regular $d$-contact stack and  $m$-regular $\Lambda$-avoiding $DLU$ paths of length $dn$ from $(0, 0)$ to $(dn, 0)$. 
$$
\eta : \mathcal{S}_{m, d}(n) \rightarrow \mathcal{G}^{\langle 0, 0\rangle }_{m, d}(dn).
$$

 For a given general stack $S\in \mathcal{S}_{m, d}(n)$, its corresponding $DLU$ path $\eta(S)$ is obtained by converting any vertex $v$ with $\mathrm{ldeg}(v)=a$ and $\mathrm{rdeg}(v)=b$ into a subpath $P _v=D^aL^{d-a-b}U^b$.

  Conversely, from a path $M=P_1P_2\ldots P_n\in\mathcal{G}^{\langle 0, 0\rangle }_{m, d}(dn)$ with $P_i=D^{a_i}L^{b_i}U^{c_i}$,  we can construct a general stack $S$ by the inverse of $\eta$  as the following steps. 
 
 \begin{enumerate}[Step 1.]
\item Draw $n$ vertices $1, 2, \ldots, n$ from left to right on a horizontal line, denoted by $S$. Let the multisets $ V_1=\{a_1\cdot 1, a_2\cdot 2, \ldots, a_n\cdot n\}, V_2=\{c_1\cdot 1, c_2\cdot 2, \ldots, c_n\cdot n\}$.
  
\item  Starting with one of the maximal vertices $i\in V_2$, connect $i$ with $j\in V_1$, where $j$ is the first vertex in $V_1$ that is larger than $i$. Add arc $(i, j)$ to $S$.
  
\item Delete $i, j$ from $V_2, V_1$, respectively.
  
\item  Return to Step 2 until $V_2$ is an empty set. 
 \end{enumerate}

As illustrated Figure \ref{fig:bij} is the corresponding $DLU$ path of Figure \ref{fig:generalstack} by $\eta$.

 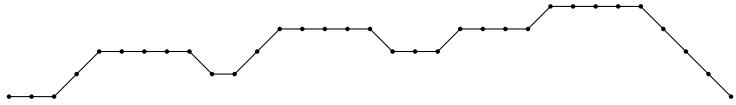
\begin{figure}[H]
\begin{center}
\setlength{\unitlength}{0.3mm}
\begin{picture}(320, 60)
  \multiput(0, 0)(10, 0){3}{\circle*{2}}        
  \put(0, 0){\line(1, 0){20}}
  \multiput(30, 10)(10, 10){2}{\circle*{2}}        
  \put(20, 0){\line(1, 1){20}}
   \multiput(50, 20)(10, 0){4}{\circle*{2}}        
  \put(40, 20){\line(1, 0){40}}
   \multiput(90, 10)(10, 0){1}{\circle*{2}}        
  \put(80, 20){\line(1, -1){10}}
   \put(90, 10){\line(1, 0){10}}
  \multiput(100, 10)(10, 10){3}{\circle*{2}}        
  \put(100, 10){\line(1, 1){20}}
   \multiput(130, 30)(10, 0){4}{\circle*{2}}        
  \put(120, 30){\line(1, 0){40}}
   \multiput(170, 20)(10, 0){3}{\circle*{2}}        
  \put(160, 30){\line(1, -1){10}}
    \put(170, 20){\line(1, 0){20}}
     \multiput(200, 30)(10, 0){4}{\circle*{2}}        
  \put(190, 20){\line(1, 1){10}}
    \put(200, 30){\line(1, 0){30}}
        \multiput(240, 40)(10, 0){5}{\circle*{2}}        
  \put(230, 30){\line(1, 1){10}}
    \put(240, 40){\line(1, 0){40}}
     \multiput(290, 30)(10, -10){4}{\circle*{2}}        
  \put(280, 40){\line(1, -1){40}}

  \end{picture}

 \caption[discard]{Corresponding $m$-regualr $\Lambda$-avoiding $DLU$ path of Figure \ref{fig:generalstack}.} 
 \label{fig:bij}
\end{center}
\end{figure}

\begin{lem}
\label{lem:bijective}
The map $\eta$ is a bijection between $m$-regular $d$-contact stacks in $[n]$ and  $m$-regular $\Lambda$-avoiding $DLU$ paths of length $dn$ from $(0, 0)$ to $(dn, 0)$ never passing below the $x$-axis.
\end{lem}

\begin{proof}

By the definition of map $\eta$, we can obtain that any general stack in $[n]$ corresponds to a $DLU$ path of length $dn$ from $(0, 0)$ to $(dn, 0)$ never passing below the $x$-axis.

For any $S\in \mathcal{S}_{m, d}(n)$, let $M=\eta(S)=P_1P_2\ldots P_n$. We claim that if $S$ is $m$-regular then $M$ is $m$-regular. If there exist an subpath $UPD$ of $M$, where $P\in\{\emptyset, ~L^d, ~L^{2d}, ~\ldots~, L^{(m-2)d}\}$, suppose the $U$ is the last up-step of $P_i$, and $D$ is the first down-step of $P_j$. As for any piece $P_k$ of path $M$, $D$ must precede $U$, it follows that $i<j$. Then $(i, j)$ is an arc of $S$ with $j-i<m$ by the corresponding $\eta^{-1}$. Then conflicting to the $m$-regular of $S$.

Conversely, assume $M$ is $m$-regular. It is clear that $S$ is $m$-regular since each arc of $S$ contains at least $m-1$ isolate vertices.

Finally, we show that $M=\eta(S)=P_1P_2\ldots P_n$ is $\Lambda$-avoiding. Suppose $M$ has pattern $\Lambda$, then there is a subpath $UUPDD$ of $M$ with $P$ starting and ending at the same height $h$ and never passing below the height $h$. Let $UU$ be two up-steps of $P_i$ and $DD$ be two down-steps of $P_j$, which corresponding to two right degrees of vertex $i$ and two left degrees of vertex $j$, respectively. Thus $(i, j)$ is a multiple arc of $S$  by the corresponding $\eta^{-1}$, contradicting with the definition of stack. This complete the proof. $\qedhere\hfill\blacksquare$
\end{proof}

\section{Enumerating method of general stacks} 
In this section, we describe a unified enumerating method of general stack with any $m,d\in\mathbb{Z}^+$ through the enumeration of $DLU$ path.
We claim that a $DLU$ path $M$ from $(0, s)$ to $(dn, t)$, where $s, t\in\mathbb{N}, s, t\leq d$ is $\Lambda$-avoiding if the $DLU$ path $L^{d-s}U^sMD^tL^{d-t}$ starting and ending at $x$-axis is $\Lambda$-avoiding. Similarly, we can describe the arbitrary $DLU$ path is $m$-regular. Denoting the set of $m$-regular $\Lambda$-avoiding $DLU$ paths of length $dn$ from $(0, s)$ to $(dn, t)$ as $\mathcal{G}^{\langle s, t\rangle }_{m, d}(dn)$, the number of $m$-regular $\Lambda$-avoiding $DLU$ paths of length $dn$ from $(0, s)$ to $(dn, t)$ as $g^{\langle s, t\rangle }_{m, d}(dn)$, and the generating function is $$G^{\langle s, t\rangle }_{m, d}(x) =\sum_{n\geq 0}g^{\langle s, t\rangle }_{m, d}
(dn)x^{dn}.$$

\begin{lem}
The number of $m$-regular $\Lambda$-avoiding $DLU$ paths of length $dn$ from $(0, s)$ to $(dn, t)$ never passing below the $x$-axis, satisfying the symmetry.
$$
g^{\langle s,t \rangle}_{m, d}(dn)=g^{\langle t, s\rangle}_{m, d}(dn)
$$
\end{lem}
\begin{proof}
 To prove symmetry, we construct the mirror reflection,
$$
\rho : \mathcal{G}^{\langle s, t \rangle}_{m, d}(dn)\rightarrow \mathcal{G}^{\langle t, s \rangle }_{m, d}(dn)
$$
satisfying that for any $M=P_1P_2\ldots P_{n}\in\mathcal{G}^{\langle s, t\rangle }_{m, d}(dn)$, $\rho(M)=\overline{P}_{n}\overline{P}_{n-1}\ldots \overline{P}_1$, where
$$
\overline{P}_i=\overline{D^{a_i}L^{b_i}U^{c_i}}=D^{c_i}L^{b_i}U^{a_i}
$$
Apperently the inverse mapping is also $\rho$.

Obviously, $\rho(M)$ is a $DLU$ path. For any $UUPDD\subset M$, then $UU\rho(P)DD\subset M$, means $M$ is $\Lambda$-avoiding if and only if $\rho(M)$ is $\Lambda$-avoiding. By the same way, we have $M$ is $m$-regular if and only if $\rho(M)$ is $m$-regular.

Thus the map $\rho$ is bijective, and so the symmetry is obtained.$\qedhere\hfill\blacksquare$
\end{proof}

To this point, we only need to consider the enumerating method for $DLU$ paths starting higher than ending. 

Recall that a Dyck path is prime \cite{deutsch1999dyck} if it never goes back to the $x$-axis in the middle steps. Similarity, we claim that a $DLU$ path of length $dn$ from $(0, s)$ to $(dn, t)$ is prime if it  never goes back to the $x$-axis in the middle steps. Denoting the set of $m$-regular $\Lambda$-avoiding prime $DLU$ paths of length $dn$ from $(0, s)$ to $(dn, t)$ as $\mathcal{A}^{\langle s, t\rangle }_{m, d}(dn)$, and its generating function as $A^{\langle s, t\rangle }_{m, d}(x)$.

\begin{lem}
\label{lem: A}
For $ 1 \leq t \leq s$, the generating function $A^{\langle s, t\rangle }_{m, d}(x)$ satisfying the following recurrence relation.
\begin{eqnarray}
\label{equ:A}
    A^{\langle s, t\rangle }_{m, d}(x)=G^{\langle s-1, t-1\rangle }_{m, d}(x)-A^{\langle s-1, t-1\rangle }_{m, d}(x),
\end{eqnarray}

where $2 \leq t \leq s$. Let $C_{m, d}(x)=\frac{x^{d(m-1)}-1}{x^d-1}$, then the initial condition is
\begin{eqnarray*}
\left\{\begin{aligned}
  &A^{\langle 1, 1\rangle }_{m, d}(x)=G^{\langle 0, 0\rangle }_{m, d}(x)-C_{m, d}(x),\\
  &A^{\langle 2, 1\rangle }_{m, d}(x)=G^{\langle 1, 0\rangle }_{m, d}(x),\\
  &A^{\langle 3, 1\rangle }_{m, d}(x)=G^{\langle 2, 0\rangle }_{m, d}(x),\\
  &\ldots\ldots\\
  &A^{\langle d, 1\rangle }_{m, d}(x)=G^{\langle d-1, 0\rangle }_{m, d}(x).\\
 \end{aligned}\right.
\end{eqnarray*}
\end{lem}
\begin{proof}
When $s=1, t=1$, since any path $M$ in $\mathcal{A}_{m, d}^{\langle 1, 1\rangle }$ does not intersect $x$-axis, then $M$ always on the height $1$. Assume that $\mathcal{C}_{m, d}=\{\emptyset, L^d, L^{2d}, \ldots, L^{(m-2)d}\}$. Let map $\phi$  represent the path moves down one step, as path is  $m$-regular, then $$\phi(\mathcal{A}_{m, d}^{\langle 1, 1\rangle })=\mathcal{G}_{m, d}^{\langle 0, 0\rangle }\backslash\mathcal{C}_{m, d}, $$
Apparently, $\phi$ is bijective between $\mathcal{A}_{m, d}^{\langle 1, 1\rangle }$ and $\mathcal{G}_{m, d}^{\langle 0, 0\rangle }\backslash\mathcal{C}_{m, d}$, hence $$A^{\langle 1, 1\rangle }_{m, d}(x)=G^{\langle 0, 0\rangle }_{m, d} (x)-C_{m, d}(x).$$

When $s=2, t=1$, similarly we have $A^{\langle 2, 1\rangle }_{m, d}(x)=G^{\langle 1, 0\rangle }_{m, d}(x)$.

When $s=2, t=2$, considering that the path avoids pattern $\Lambda$, then $\forall M\in\mathcal{A}^{\langle 2, 2\rangle }_{m, d}$, we have $M$ must reach height $1$ but not height $0$. Hence, $\phi(\mathcal{A}^{\langle 2, 2\rangle }_{m, d})$ is the subset of $\mathcal{G}_{m, d}^{\langle 1, 1\rangle } $ which must intersect the $x$-axis, means
$$\phi(\mathcal{A}^{\langle 2, 2\rangle }_{m, d})=\mathcal{G}^{\langle 1, 1\rangle }_{m, d}-\mathcal{A}^{\langle 1, 1\rangle }_{m, d}.$$

Using mathematical induction, assume that for any $i, j$, where $ 2<i<j<t<s$, holds that $A^{\langle i, j\rangle }_{d}(x)=G^{\langle i-1, j-1\rangle }_{d}(x)-A^{\langle i-1, j-1\rangle }_{d}(x)$.

Let arbitrary $ M \in \mathcal{A}^{\langle s, t\rangle }_d$, we have $M$ must reach height $1$ but not height $0$, thus  $\phi(\mathcal{A}^{\langle s, t\rangle }_{m, d})$ is the subset of $\mathcal{G}_{m, d}^{\langle s-1, t-1\rangle } $ which must intersect the $x$-axis, means
$$\phi(\mathcal{A}^{\langle s, t\rangle }_{m, d})=\mathcal{G}^{\langle s-1, t-1\rangle }_{m, d}-\mathcal{A}^{\langle s-1, t-1\rangle }_{m, d}.$$

Therefore, generating function $A^{\langle s, t\rangle }_{m, d}(x)$ satisfies the following equation.$\qedhere\hfill\blacksquare$
\end{proof}

\begin{proof}[Proof of Theorem \rm{\ref{thm:syseq}}]
According to Lemma \ref{lem:bijective}, we can directly obtain equation \eqref{equ:samenum}. we distinguish the discussion according to the values of $s$ and $t$ to focus on the enumeration of  $DLU$ paths  $\mathcal{G}^{\langle 0, 0\rangle }_{m, d}$.

Firstly, we consider the enumeration of $\mathcal{G}^{\langle 0, 0\rangle }_{m, d}$. For any path $M=P_1P_2\ldots\in \mathcal{G}^{\langle 0, 0\rangle }_{m, d}$, if the length of $M$ is $0$, then it corresponds to only one empty path, means $G^{\langle 0, 0\rangle }_{m, d}(0)=1$. If the length of $M$ is greater than $0$, consider the first piece 
$P_1=D^0L^bU^c$ with $b+c=d, 0\leq b, c\leq d$ for the path $M$. Since the path is never passing below the $x$-axis, then the remaining $P_2P_3\ldots$ is a path in $\mathcal{G}^{\langle c, 0\rangle }_{m, d}$. Hence we have
\begin{eqnarray}
\label{equ:g00}
G^{\langle 0, 0\rangle }_{m, d}(x)=1+x^d\sum\limits_{c=0}^{d}G^{\langle c, 0\rangle }_{m, d}(x).
\end{eqnarray}

Next, we consider the enumeration of $\mathcal{G}^{\langle s, t\rangle }_{m, d}$, where $t \leq s$ and $s, t$ are not simultaneously $0$.  For arbitrary $M=P_1P_2P_3\ldots\ldots \in \mathcal{G}^{\langle s, t\rangle }_{m, d} $, if $M$ has no intersection with the $x-$axis, then $t\geq 1$,  means that $M\subset \mathcal{A}^{\langle s, t\rangle }_{m, d}$. By calculating the recurrence formula \eqref{equ:A}, we can derive
\begin{eqnarray}
\label{equ:Aexpand}
A^{\langle s, t\rangle }_{m, d}(x)=\delta_{s, t}(-1)^tC_{m, d}+\sum\limits_{i=0}^{t-1}(-1)^{t-i+1}G^{\langle s-t+i, i\rangle }_{m, d}.
\end{eqnarray}
If $M$ must intersect the $x$-axis, consider the mothod of first return decomposition \cite{deutsch1999dyck}, then let the first piece $P_k=D^{a_k}L^{b_k}U^{c_k}$ returning to the $x$-axis, we have $a_k\geq 1$. Note that we get a path decomposition $M=M_1P_kM_2$, means that $M_1, M_2$ can be builded by $a_k, c_k$.
\begin{enumerate}
    \item For a preorder path $M_1=P_1P_2\ldots P_{k-1}$ of $P_k$, we find that $M_1$ is a prime path from height $s$ to height $a_k$, i.e. it is path in $\mathcal{A}^{\langle s, a_k\rangle }_{m, d}$. 

   \item For the subsequence path $M_2=P_{k+1}P_{k+2}\ldots\ldots$ of $P_k$, we have $M_2$ is the paths from height $c_k$ to height $t$, then  it is path in $\mathcal{A}^{\langle c_k, t\rangle }_{m, d}$.  
\end{enumerate}
In summary, we can obtain that the generating function of $\mathcal{G}^{\langle s, t\rangle }_{m, d}$ satisfies the follwing equation,
\begin{eqnarray}
\label{equ:gst}
G^{\langle s, t\rangle }_{m, d}(x)=A^{\langle s, t\rangle }_{m, d}(x)+\sum\limits_{
a_k+c_k\leq d, 
\atop
a_k\geq 1 
}\Big[A^{\langle u, v\rangle }_{m, d}(x)\cdot x^d \cdot G^{\langle c_k, t\rangle }_{m, d}(x)\Big], 
\end{eqnarray}
where $u=\max(s, a_k), v=\min(s, a_k)$.
  
Substituting equation \eqref{equ:Aexpand} into equation \eqref{equ:gst}, we can obtained the equation of $G^{\langle s, t\rangle }_{m, d}(x)$ for $0\leq t \leq d-1, ~0<s\leq d, ~t<s$. Hence, we can combine with the above equations of  $G^{\langle s, t\rangle }_{m, d}(x)$ and equation \eqref{equ:g00} to derive the system of equations \eqref{equ:syseq}.$\qedhere\hfill\blacksquare$
\end{proof}

\section{Special cases} 
In this section, the enumerating results on $d=1, 2, 3$ are given by the counting methods of the previous section for general stacks.

A simple stack is a $1$-contact stack. According to  Theorem \ref{thm:syseq}, we can list the following system of equations of the generating function, which relating the $DLU$ path corresponding the $m$-regular simple stacks.
\begin{eqnarray*}
\left\{\begin{aligned}
  &G^{\langle 0, 0\rangle }_{m, 1}=1+x^2\Big(G^{\langle 0, 0\rangle }_{m, 1}+G^{\langle 1, 0\rangle }_{m, 1}\Big)\\
  &G^{\langle 1, 0\rangle }_{m, 1}=x^2\Big(G^{\langle 0, 0\rangle }_{m, 1}-C_{m, 1}\Big)G^{\langle 0, 0\rangle }_{m, 1}
    \end{aligned}\right.
\end{eqnarray*}
where $C_{m, 1}(x)=\frac{x^{m-1}-1}{x-1}$.

Solving by Maple yields that
\begin{eqnarray*}
(x^3-x^2)\bigg[G^{\langle 0, 0\rangle }_{m, 1}(x)\bigg]^2+(-x^{m+1}+2x^2-2x+1)G^{\langle 0, 0\rangle }_{m, 1}(x)+x-1=0.
\end{eqnarray*}
 Depending on equation \eqref{equ:samenum}, we have
 \begin{eqnarray*}
S_{m, 1}(x)=\frac{\sum\limits_{l=0}^mx^l-2x-\sqrt{\bigg(\sum\limits_{l=0}^mx^l-2x\bigg)^2-4x^2}}{2x^2}
\end{eqnarray*}
When $m=1$, we have
$$S_{1, 1}(x)=\frac{1-x-\sqrt{1-2x-3x^2}}{2x^2}, $$
which is the generating function of $Motzkin$ number.
\begin{table}[h]
\addtocounter{table}{0}
 \vspace{-5pt}
 \renewcommand{\tablename}{Table}
 \caption{Enumeration of $m$-regular simple stack}
 \label{table:simple}
    \centering
    \begin{tabular}{lcccccccccc}
    \hline
        $n$ & 1 & 2 & 3 & 4 & 5 & 6 & 7 & 8 & 9 & 10 \\\hline
        $s_{1, 1}(n)$  & 1 & 2 & 4 & 9 & 21 & 51 & 127 & 323 & 835 &2188 \\
        $s_{2, 1}(n)$  & 1 & 1 & 2 & 4 & 8 & 17 & 37 & 82 & 185  &423\\
        $s_{3, 1}(n)$  & 1 & 1 & 1 & 2 & 4 & 8 & 16 & 33 & 69 &146 \\
        $s_{4, 1}(n)$  & 1 & 1 & 1 & 1 & 2 & 4 & 8 & 16 & 32 &65 \\
        $s_{5, 1}(n)$ & 1 & 1 & 1 & 1 & 1 & 2 & 4 & 8 & 16 &32 \\
        $s_{6, 1}(n)$  & 1 & 1 & 1 & 1 & 1 & 1 & 2 & 4 & 8 &16  \\\hline
    \end{tabular}
\end{table}

The enumeration results shown in Table \ref{table:simple} are commit with those of $m$-regular simple stacks counted by 
Stein et al.\cite{stein1979rna}, who derived the same generating function.

A linear stack is a $2$-contact stack. Similarly enumerating method as for the simple stack, we have
\begin{eqnarray*}
\left\{\begin{aligned}
  G^{\langle 0, 0\rangle }_{m, 2}&=1+x^2\sum\limits_{i=0}^{2}G^{\langle i, 0\rangle }_{m, 2}\\
G^{\langle 1, 0\rangle }_{m, 2}&=x^2\bigg[\Big(G^{\langle 0, 0\rangle }_{m, 2}-C_{m, 2}\Big)\sum\limits_{i=0}^{1}G^{\langle i, 0\rangle }_{m, 2}+G^{\langle 1, 0\} }_{m, 2}G^{\langle 0, 0\rangle }_{m, 2}\bigg]\\
  G^{\langle 1, 1\rangle }_{m, 2}&=G^{\langle 0, 0\rangle }_{m, 2}-C_{m, 2}+x^2\bigg[\Big(G^{\langle 0, 0\rangle }_{m, 2}-C_{m, 2}\Big)\sum\limits_{i=0}^{1}G^{\langle i, 1\rangle }_{m, 2}+G^{\langle 1, 0\rangle }_{m, 2}G^{\langle 0, 1\rangle }_{m, 2}\bigg]\\
  G^{\langle 2, 0\rangle }_{m, 2}&=x^2\bigg[G^{\langle 1, 0\rangle }_{m, 2}\sum\limits_{i=0}^{1}G^{\langle i, 0\rangle }_{m, 2}+\Big(G^{\langle 1, 1\rangle }_{m, 2}-G^{\langle 0, 0 \rangle }_{m, 2}+C_{m, 2}\Big)G^{\langle 0, 0\rangle }_{m, 2}\bigg]
  \end{aligned}\right.
\end{eqnarray*}
where $C_{m, 2}(x)=\frac{x^{2(m-1)}-1}{x^2-1}$.

Solving by Maple yields that
\begin{eqnarray*}
\sum\limits_{i=0}^{5}c_{m, 2, i}(x^2)\Big[G^{\langle 0, 0\rangle }_{m, 2}(x)\Big]^i=0.
\end{eqnarray*}
then
\begin{eqnarray*}
\sum\limits_{k=0}
^{5}c_{m, 2, k}(x)S_{m, 2}^k(x)=0, 
\end{eqnarray*}
where
\begin{eqnarray*}
\begin{aligned}
c_{m, 2, 5}(x)&=4 x^{5} (x - 1)^{4}, \\
c_{m, 2, 4}(x)&=-11 x^{m +7}+4 x^{8}+33 x^{m +6}-x^{7}-33 x^{m +5}-25 x^{6}+11 x^{m +4}\\&+41 x^{5}-23 x^{4}+4 x^{3}, \\
c_{m, 2, 3}(x)&= 11 x^{2 m +5}-8 x^{m +6}-22 x^{2 m +4}-6 x^{m +5}+11 x^{2 m +3}+8 x^{6}\\&+44 x^{m +4}-5 x^{5}-38 x^{m +3}-22 x^{4}+8 x^{m +2}+27 x^{3}-8 x^{2}, \\
c_{m, 2, 2}(x)&=-5 x^{3 m +3}+6 x^{2 m +4}+5 x^{3 m +2}+3 x^{m +5}+8 x^{2 m +3}\\&-21 x^{m +4}-19 x^{2 m +2}-3 x^{5}+8 x^{m +3}+5 x^{2 m +1}\\&+15 x^{4}+20 x^{m +2}-11 x^ {3}-10 x^{1+m}-6 x^{2}+5 x, \\
c_{m, 2, 1}(x)&=x^{4 m +1}-2 x^{3 m +2}-3 x^{2 m +3}-3 x^{3 m +1}+12 x^{2 m +2}+x^{3 m}\\&+6 x^{m +3}-18 x^{m +2}-3 x^{2 m}-3 x^{3}+5 x^{1+m}+8 x^{2}+3 x^{m}-3 x -1, \\
c_{m, 2, 0}(x)&=x^{3 m +1}-x^{3 m}-3 x^{2 m +1}+3 x^{2 m}+3 x^{1+m}-3 x^{m}-x +1.
\end{aligned}
\end{eqnarray*}

\begin{table}[H]
\addtocounter{table}{0}
 \vspace{-5pt}
 \renewcommand{\tablename}{Table}
 \caption{Enumeration of $m$-regular linear stack}
 \label{table:linear}
    \centering
    \begin{tabular}{lcccccccccc}
    \hline
        $n$ & 1 & 2 & 3 & 4 & 5 & 6 & 7 & 8 & 9 & 10 \\\hline
$s_{1, 2}(n)$ & 1 & 2 & 8 & 34 & 147 & 663 & 3096 & 14814 & 72227 & 357591 \\
        $s_{2, 2}(n)$ & 1 & 1 & 2 & 6 & 20 & 66 & 221 & 757 & 2647 & 9402 \\
        $s_{3, 2}(n)$ & 1 & 1 & 1 & 2 & 6 & 18 & 54 & 162 & 491 & 1509 \\
        $s_{4, 2}(n)$ & 1 & 1 & 1 & 1 & 2 & 6 & 18 & 52 & 150 & 434 \\
        $s_{5, 2}(n)$ & 1 & 1 & 1 & 1 & 1 & 2 & 6 & 18 & 52 & 148 \\
        $s_{6, 2}(n)$ & 1 & 1 & 1 & 1 & 1 & 1 & 2 & 6 & 18 & 52 \\\hline
    \end{tabular}
\end{table}
The enumeration results shown in Table \ref{table:linear} are consistent with those of $m$-regular linear stacks counted by Chen et al. in 2014 {\cite{Chen2014}}, who use the $reduction$ method and the $zigzag$ structure.

Finally, We are able to obtain an equation that the generating function of $m$-regular $3$-contact stack satisfying.

\begin{cor}
\label{cor:spatial}
The generating function for the $m$-regular $3$-contact stack satisfies the equation
\begin{eqnarray*}
\sum\limits_{k=0}
^{17}c_{m, 3, k}(x)S_{m, 3}^k(x)=0.
\end{eqnarray*}
where the coefficients $c_{m, 3, k}(x)$ are detailed at 
\href{https://github.com/SiennaXSN/General-Stack}{https://github.com/SiennaXSN/General-Stack}.

\end{cor}
\begin{proof}
Similarity to the enumerating method of simple stack, we can obtain the following set of equations, 
\begin{eqnarray*}
\left\{\begin{aligned}
  G^{\langle 0, 0\rangle }_{m, 3}&=1+x^3\sum_{i=0}^3 G^{\langle i, 0\rangle }_{m, 3}\\
  G^{\langle 1, 0\rangle }_{m, 3}&=x^3\bigg[\Big(G^{\langle 0, 0\rangle }_{m, 3}-C_{m, 3}\Big)\sum_{i=0}^2 G^{\langle i, 0\rangle }_{m, 3}+G^{\langle 1, 0\rangle }_{m, 3}\sum_{i=0}^1 G^{\langle i, 0\rangle }_{m, 3}+G^{\langle 2, 0\rangle }_{m, 3}G^{\langle 0, 0\rangle }_{m, 3}\bigg]\\
  G^{\langle 1, 1\rangle }_{m, 3}&=G^{\langle 0, 0\rangle }_{m, 3}-C_{m, 3}+x^3\bigg[\Big(G^{\langle 0, 0\rangle }_{m, 3}-C_{m, 3}\Big)\sum_{i=0}^2 G^{\langle i, 1\rangle }_{m, 3}+G^{\langle 1, 0\rangle }_{m, 3}\sum_{i=0}^1 G^{\langle i, 1\rangle }_{m, 3}\\&+G^{\langle 2, 0\rangle }_{m, 3}G^{\langle 0, 1\rangle }_{m, 3}\bigg]\\
  G^{\langle 2, 0\rangle }_{m, 3}&=x^3\bigg[G^{\langle 1, 0\rangle }_{m, 3}\sum_{i=0}^2 G^{\langle i, 0\rangle }_{m, 3}+\Big(G^{\langle 1, 1\rangle }_{m, 3}-G^{\langle 0, 0\rangle }_{m, 3}+C_{m, 3}\Big)\sum_{i=0}^1 G^{\langle i, 0\rangle }_{m, 3}\\&+\Big(G^{\langle 2, 1\rangle }_{m, 3}-G^{\langle 1, 0\rangle }_{m, 3}\Big)G^{\langle 0, 0\rangle }_{m, 3}\bigg]\\
  G^{\langle 2, 1\rangle }_{m, 3}&=G^{\langle 1, 0\rangle }_{m, 3}+x^3\bigg[G^{\langle 1, 0\rangle }_{m, 3}\sum_{i=0}^2 G^{\langle i, 1\rangle }_{m, 3}\\&+\Big(G^{\langle 1, 1\rangle }_{m, 3}-G^{\langle 0, 0\rangle }_{m, 3}+C_{m, 3}\Big)\sum_{i=0}^1 G^{\langle i, 1\rangle }_{m, 3}+\Big(G^{\langle 2, 1\rangle }_{m, 3}-G^{\langle 1, 0\rangle }_{m, 3}\Big)G^{\langle 0, 1\rangle }_{m, 3}\bigg]\\
  G^{\langle 2, 2\rangle }_{m, 3}&=G^{\langle 1, 1\rangle }_{m, 3}-G^{\langle 0, 0\rangle }_{m, 3}+C_{m, 3}+x^3\bigg[G^{\langle 1, 0\rangle }_{m, 3}\sum_{i=2}^2 G^{\langle i, 2\rangle }_{m, 3}\\&+\Big(G^{\langle 1, 1\rangle }_{m, 3}-G^{\langle 0, 0\rangle }_{m, 3}+C_{m, 3}\Big)\sum_{i=0}^1 G^{\langle i, 2\rangle }_{m, 3}+\Big(G^{\langle 2, 1\rangle }_{m, 3}-G^{\langle 1, 0\rangle }_{m, 3}\Big)G^{\langle 0, 2\rangle }_{m, 3}\bigg]\\
  G^{\langle 3, 0\rangle }_{m, 3}&=x^3\bigg[G^{\langle 2, 0\rangle }_{m, 3}\sum_{i=0}^2 G^{\langle i, 0\rangle }_{m, 3}+\Big(G^{\langle 2, 1\rangle }_{m, 3}-G^{\langle 1, 0\rangle }_{m, 3}\Big)\sum_{i=0}^1 G^{\langle i, 0\rangle }_{m, 3}\\&+\Big(G^{\langle 2, 2\rangle }_{m, 3}-G^{\langle 1, 1\rangle }_{m, 3}+G^{\langle 0, 0\rangle }_{m, 3}-C_{m, 3}\Big)G^{\langle 0, 0\rangle }_{m, 3}\bigg]
  \end{aligned}\right.
\end{eqnarray*}
where $C_{m, 3}(x)=\frac{x^{3(m-1)}-1}{x^3-1}$.

Solving by Maple gives us that
\begin{eqnarray*}
\sum\limits_{k=0}^{17}c_{m, 3, k}(x^3)\bigg[G^{\langle 0, 0\rangle }_{m, 3}(x)\bigg]^k=0.
\end{eqnarray*}$\qedhere\hfill\blacksquare$
\end{proof}

when $m=1$, we have
\begin{eqnarray*}
\sum\limits_{k=0}
^{17}c_{1, 3, k}(x)S_{1, 3}^k(x)=0.
\end{eqnarray*}
where
\begin{eqnarray*}
\begin{aligned}
c_{1, 3, 0} &= 1 \\
c_{1, 3, 1}&= x \\
c_{1, 3, 2}&= 52 x^{2}+x^{1}-1 \\
c_{1, 3, 3}&= 60 x^{3}-5 x^{2} \\
c_{1, 3, 4}&= 787 x^{4}-115 x^{3}-45 x^{2} \\
c_{1, 3, 5}&= 871 x^{5}-376 x^{4}-3 x^{3} \\
c_{1, 3, 6}&= 5731 x^{6}-511 x^{5}-1314 x^{4}+162 x^{3} \\
c_{1, 3, 7}&= 8188 x^{7}-7823 x^{6}+315 x^{5} \\
c_{1, 3, 8}&= 21690 x^{8}-4575 x^{7}-5483 x^{6}-702 x^{5}+729 x^{4} \\
c_{1, 3, 9}&= 35452 x^{9}-37272 x^{8}-10146 x^{7}+7938 x^{6} \\
c_{1, 3, 10}&= 53179 x^{10}-50388 x^{9}+11385 x^{8}+3888 x^{7} \\
c_{1, 3, 11}&= 63508 x^{11}-84400 x^{10}+11576 x^{9}+6912 x^{8} \\
c_{1, 3, 12}&= 65208 x^{12}-79736 x^{11}+39360 x^{10} \\
c_{1, 3, 13}&= 53304 x^{13}-69312 x^{12}+18432 x^{11} \\
c_{1, 3, 14}&= 34304 x^{14}-38912 x^{13}+16384 x^{12} \\
c_{1, 3, 15}&= 10240 x^{15}-16384 x^{14} \\
c_{1, 3, 16}&= c_{1, 3, 17}=0
\end{aligned}
\end{eqnarray*}

According to Corollary \ref{cor:spatial}, we list following Table \ref{table:spatial} of some enumerating number of $m$-regular $3$-contact stack.
\begin{table}[H]
\addtocounter{table}{0}
 \vspace{-5pt}
 \renewcommand{\tablename}{Table}
 \caption{Enumeration of $m$-regular $3$-contact stack}
 \label{table:spatial}
    \centering
    \begin{tabular}{lcccccccccc}
    \hline
    $n$ & 1 & 2 & 3 & 4 & 5 & 6 & 7 & 8 & 9 & 10 \\ \hline
$s_{1, 3}(n)$ & 1 & 2 & 8 & 48 & 312 & 2062 & 13890 & 95558 & 669842 & 4768645 \\
$s_{2, 3}(n)$ & 1 & 1 & 2 & 6 & 22 & 88 & 364 & 1534 & 6561 & 28445 \\
$s_{3, 3}(n)$ & 1 & 1 & 1 & 2 & 6 & 20 & 68 & 236 & 832 & 2970 \\
$s_{4, 3}(n)$ & 1 & 1 & 1 & 1 & 2 & 6 & 20 & 66 & 216 & 710 \\
$s_{5, 3}(n)$ & 1 & 1 & 1 & 1 & 1 & 2 & 6 & 20 & 66 & 214 \\
$s_{6, 3}(n)$ & 1 & 1 & 1 & 1 & 1 & 1 & 2 & 6 & 20 & 66 \\\hline
    \end{tabular}
\end{table}
 
At last, we illustrate the growing curves of $s_{m, d}(n)$ for $m=2,5$ and $d=1,2,3$ in Figure \ref{fig:trend}.

\begin{figure}[H]
\begin{minipage}[t]{1\textwidth}
\centering
\includegraphics[scale=0.33]{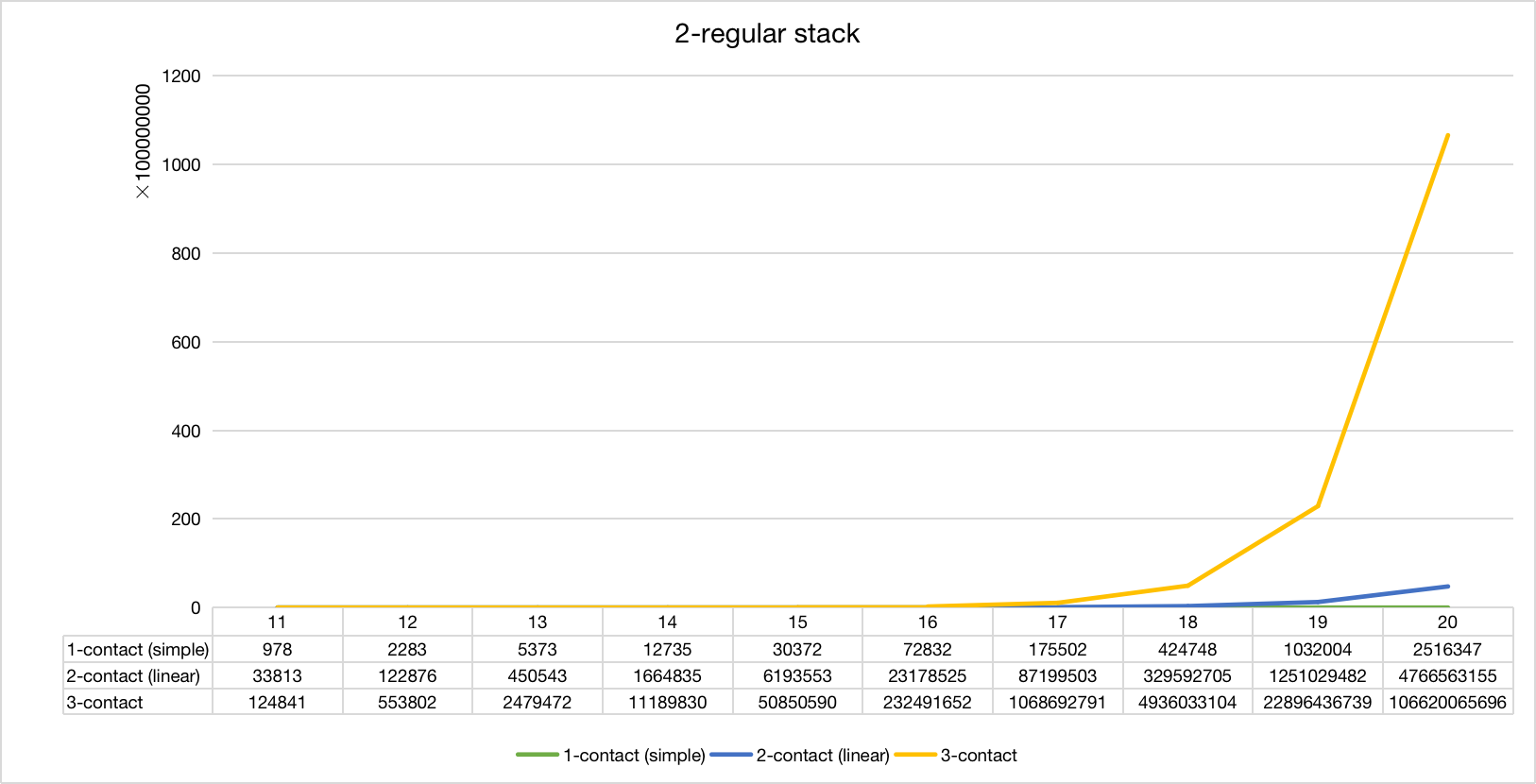}
\centerline{(a) $2$-regular stack}
\label{fig:side:a}
\end{minipage}%

\begin{minipage}[t]{1\textwidth}
\centering
\includegraphics[scale=0.33]{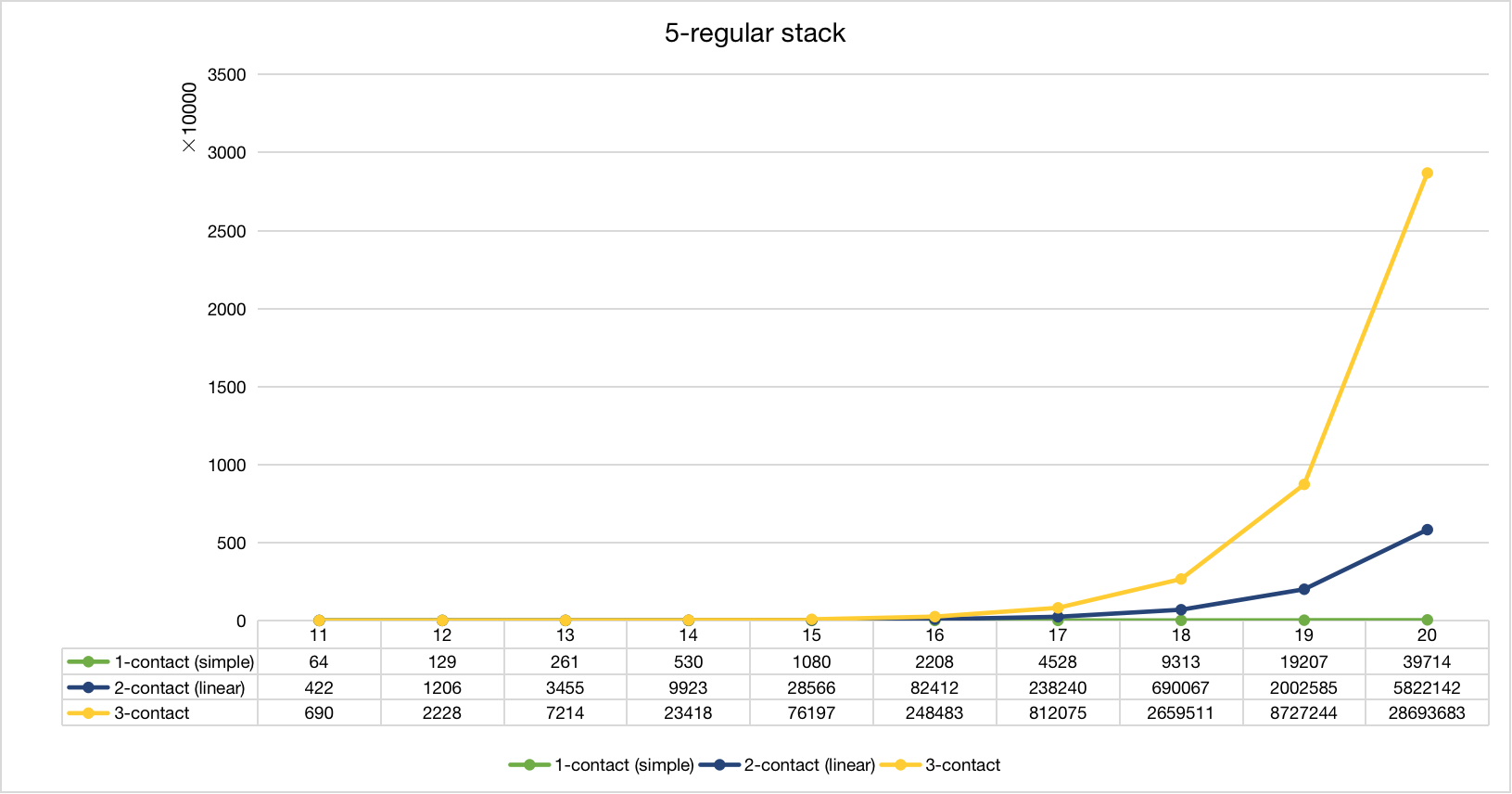}
\centerline{(b) $5$-regular stack}
\end{minipage}
\caption{Growing curves of $s_{m,d}(n)$}
\label{fig:trend}
\end{figure}

\section*{Acknowledgments}
This work was supported by the National Natural Science Foundation of China (Grant No. 12071235 and 11501307), and the Fundamental Research Funds for the Central Universities.

\end{document}